%% file: domnigamsutton_v1Nov2019_arxiv.tex
\documentclass[final,11pt]{article}
\usepackage[utf8]{inputenc}
\usepackage{anysize}
\marginsize{3cm}{3cm}{2cm}{2cm} 

\usepackage{fancyhdr} 
\usepackage{lipsum} 
\usepackage{color} 
\usepackage{bm} 
\usepackage{amsfonts,amsthm,amsmath,amssymb} 
\usepackage{latexsym,graphicx,epstopdf} 
\usepackage{hyperref}
\usepackage{etoolbox} 
\usepackage[left,mathlines]{lineno}

\renewcommand{\a}{{\bf a}}
\renewcommand{\b}{{\bf b}}

\newcommand{\f}{{\bf f}}
\newcommand{\g}{{\bf g}}

\newcommand{\n}{{\bf n}}

\renewcommand{\r}{{\bf r}}
\newcommand{\s}{{\bf s}}
\renewcommand{\t}{{\bf t}}
\renewcommand{\u}{{\bf u}}
\renewcommand{\v}{{\bf v}}

\newcommand{\x}{{\bf x}}

\newcommand{\zero}{{\bf 0}}

\newcommand{\bb}{{\bf B}}

\newcommand{\hh}{{\bf H}}
\newcommand{\ii}{{\bf I}}

\renewcommand{\ll}{{\bf L}}
\newcommand{\mm}{{\bf M}}

\newcommand{\rr}{{\bf R}}

\renewcommand{\tt}{{\bf T}}

\newcommand{\ww}{{\bf W}}

\newcommand{\ccc}{\mathbb{C}}

\newcommand{\hhh}{\mathbb{H}}

\renewcommand{\lll}{\mathbb{L}}

\newcommand{\nnn}{\mathbb{N}}

\newcommand{\rrr}{\mathbb{R}}

\newcommand{\btau}{{\bm\tau}}

\newcommand{\bsigma}{{\bm\sigma}}
\newcommand{\bepsilon}{{\bm\epsilon}}

\renewcommand{\div}{{\rm div}}
\newcommand{\bfdiv}{{\bf div}}
\newcommand{\curl}{{\rm curl}}

\newcommand{\tr}{{\rm tr}}
\newcommand{\transpose}{\texttt{t}}

\newcommand{\gen}{{\rm span}}

\newtheorem{theorem}{Theorem}
\newtheorem{definition}{Definition}

\title{On Korn's inequality and the Jones eigenproblem on Lipschitz domains\thanks{This work was partially 
supported by CONICYT-Chile, through Becas Chile, and NSERC
through the Discovery program of Canada.}}

\author{Sebasti\'an Dom\'inguez\thanks{Department of Mathematics, Simon Fraser University, Burnaby, BC, 
Canada}{\,\,\,\thanks{Corresponding author: \href{mailto:domingue@sfu.ca}{domingue@sfu.ca}}}
\quad Nilima Nigam\footnotemark[2]} 

\begin{document}

\maketitle

\begin{abstract}
\input{abstract.tex}
\end{abstract}


{\bf Keywords}: Korn's inequality, linear elasticity, Jones eigenvalue problem
\vspace{.25cm}

{\bf AMS subject classifications}: 47A75, 74B05, 74F10

\section{Introduction}\label{section:intro}
\input{introduction.tex}

\section{Korn's inequality for Lipschitz domains}\label{section:korns}
\input{korn.tex}

\section{The Jones eigenvalue problem}\label{section:jones}
\input{jones.tex}

\paragraph{Conclusions}
In this manuscript we have studied the properties of the so-called {\it Jones eigenvalue problem} on Lipschitz domains. To this end, a new Korn's inequality for smooth enough vector fields with vanishing normal trace was proved whenever the domain $\Omega$ is Lipschitz. We were able to show this inequality even in the case one assumes the boundary condition is only prescribed on a subset of the boundary with positive $(n-1)$-dimensional measure, $\Sigma$. However, in order to obtain the Korn's inequality for such vector fields one needs to make assumptions on the geometrical properties of $\Sigma$ (cf. \autoref{result:rmnormal}). A similar conclusion is provided for vector fields with a zero tangential trace on $\Sigma$; in this case the geometry of $\Sigma$ must be constrained differently (cf. \autoref{result:rmtangent}). For both cases of the Korn's inequality we are able to extend the inequality for vector fields in the Sobolev space $\ww^{1,p}(\Omega)$. These inequalities were utilized to show that the Jones eigenproblem possesses a countable spectrum on bounded Lipschitz domains. More generally, we considered the eigenproblem where the vanishing normal trace is assumed only on $\Sigma$; this case also has a countable set of eigenpairs for such class of domains. In addition, we proved that a variant of the Jones eigenproblem, where the zero tangential trace replaces the zero normal trace on $\Sigma$, also has a countable set of eigenpairs in $\hh^1_\s(\Omega;\Sigma)$. Finally, we see that the properties of the spectrum do not change when a variable density  elastic body is considered, with the orthogonality of the eigenfunctions associated with different eigenvalues established in the appropriated weighted inner product.


\paragraph{Acknowledgements}
Sebasti\'an Dom\'inguez thanks the support of CONICYT-Chile, through Becas Chile. Nilima Nigam thanks the support of 
NSERC through the Discovery program of Canada.

\bibliographystyle{plain}
\bibliography{references.bib}

\end{document}

%% file: abstract.tex

In this paper we show that  Korn's inequality \cite{ref:korn1906} holds for vector fields with a 
zero 
normal  or tangential trace on a subset (of positive measure) of the boundary of Lipschitz domains. We further show that the validity of this 
inequality depends on the geometry of this subset of the boundary. 
We then consider the {\it Jones eigenvalue problem} which consists of  the usual traction eigenvalue problem for the Lam\'e operator for linear elasticity coupled with 
a zero normal trace of the displacement on a non-empty part of the boundary.
Here we extend the theoretical results in \cite{ref:bauer2016-1,ref:bauer2016-2,ref:dominguez2019} to show the Jones eigenpairs  exist on a broad variety of domains even when the normal trace of the displacement is constrained only on a subset of the boundary. We further show that one can have eigenpairs of a modified eigenproblem in which the constraint on the normal trace is replaced by one on the tangential trace.



%% file: introduction.tex
 Korn's inequality was first introduced in a pioneering work by Arthur Korn in 1906 \cite{ref:korn1906}. For an open and bounded domain $\Omega$ of $\rrr^n$, $n\geq 2$, A. Korn showed the existence of a positive constant $C>0$ such that

\begin{align}
    \|\nabla\u\|_{0,\Omega} \leq C\|\bepsilon(\u)\|_{0,\Omega},\label{eq:introfirstkorn}
\end{align}
for any vector field $\u:=(u_1,\ldots,u_n)^\transpose$ in $[H^1(\Omega)]^n$ subject to a zero boundary condition along the boundary of $\Omega$. The space $[H^1(\Omega)]^n$ denotes the vector version of the usual Hilbert space $H^1(\Omega)$ for functions in $L^2(\Omega)$ such that each first order derivative belongs to $L^2(\Omega)$, and $\|\cdot\|_{0,\Omega}$ being the usual $L^2$-norm applied to vector or tensor fields. Here $\bepsilon(\u)$ is the strain tensor or the symmetric part of the tensor $\nabla\u$. This inequality is usually referred to as the {\it Korn's first inequality}. In a second publication \cite{ref:korn1909}, A. Korn proved that the inequality in \autoref{eq:introfirstkorn} also holds for vector fields $\u := (u_1,u_2)^\transpose$ in $[H^1(\Omega)]^2$ satisfying the free-rotation condition

\begin{align*}
    \int_\Omega \left(\frac{\partial u_1}{\partial x_2} - \frac{\partial u_2}{\partial x_1}\right) = 0.
\end{align*}
This version of \autoref{eq:introfirstkorn} is known as  {\it Korn's second inequality}.

Note that \autoref{eq:introfirstkorn} cannot hold for arbitrary vector field in $[H^1(\Omega)]^n$. The inequality is violated for the so-called {\it rigid motions}, which are vector fields with strain-free energy. Indeed, one can see that $\bepsilon(\cdot)$ defines a linear and bounded operator in $[H^1(\Omega)]^n$ whose kernel exactly coincides with the space of all rigid motions. We then see that the zero boundary condition or the rotation free condition above are simply two different ways of avoiding these rigid motions. This motivates us to think about other ways of constraining  vector fields in $[H^1(\Omega)]^n$ while still satisfying Korn's inequality in \autoref{eq:introfirstkorn} with a finite constant. For  
example, if tangential or normal components of the vector fields are zero on the boundary of the domain, then certain domains still support rigid motions. In \cite{ref:desvillettes2002}, it was proven 
that the Korn's inequality in \autoref{eq:introfirstkorn} holds for $C^2$ non-axisymmetric domains when a vanishing normal trace of the vector field is assumed on the boundary. Later, authors in \cite{ref:bauer2016-2} 
extended this result for non-axisymmetric Lipschitz domains and additionally proved that the same inequality holds (perhaps with a different constant) when the tangential trace of the vector fields is zero along the boundary. In this case however, the shape of the boundary does not need to be constrained.

 In the present work we show that the Korn's inequality 
in \autoref{eq:introfirstkorn} remains valid even when the normal trace or tangential trace of smooth enough vector fields vanish only on a subset of the boundary with positive $(n-1)$-dimensional measure. Specifically, we show the existence of a constant $c_\Sigma>0$ such that

\begin{align*}
    \|\u\|_{0,\Omega} + \|\nabla\u\|_{0,\Omega} \leq\, c_\Sigma\|\bepsilon(\u)\|_{0,\Omega},
\end{align*}
for vector fields $\u$ in $[H^1(\Omega)]^n$. However, as shall be seen, there are many cases to watch out for to prevent
rigid motions: flat faces can support orthogonal translations which form part of the kernel of the strain tensor. As shown in 
\cite{ref:bauer2016-2} this is not the case when the normal or tangential traces are zero on the entire boundary. Only 
rotations are part of the null space of the strain tensor whenever the zero normal trace is placed on the boundary of an 
axisymmetric Lipschitz domain. In constrast the strain tensor becomes injective if the zero tangential trace is put on the 
boundary of a Lipschitz domain, with no extra assumptions on the shape of the boundary.

We are also interested in studying the following eigenvalue problem: find displacements $\u$ of an isotropic elastic body $\Omega$ of $\rrr^n$, $n\geq 2$, with Lipschitz boundary $\partial\Omega$ and frequencies $\omega\in\ccc$ satisfying the eigenproblem:

\begin{subequations}\label{eq:introjones}
\begin{align}
    \bsigma(\u) := 2\mu\bepsilon(\u) + \lambda\,\tr(\bepsilon(\u))\ii\quad\text{in $\Omega$},\label{eq:introjones1}\\ -\bfdiv\bsigma(\u) = \rho \omega^2 \u\quad\text{in $\Omega$},\label{eq:introjones2}\\
    \bsigma(\u)\n = \zero,\,\,\u\cdot\n = 0\quad\text{on $\partial\Omega$}.\label{eq:introjones3}
\end{align}
\end{subequations}
Here $\mu$ and $\lambda$ are the usual Lam\'e parameters, $\rho>0$ is the density of the material in $\Omega$, $\bsigma(\u)$ is the Cauchy tensor and $\n$ stands for the outward normal unit vector on $\partial\Omega$. Eigenpairs (respectively eigenvalues or eigenfunctions) solving solving this problem are called {\it Jones eigenpairs} (respectively eigenvalues or eigenfunctions).

The eigenproblem defined by \autoref{eq:introjones} is known as {\it the Jones eigenvalue problem}, first introduced by D.S. Jones in \cite{ref:jones1983}. Here the author considered a fluid-structure interaction problem where a bounded and isotropic elastic body is immersed in an unbounded inviscid compressible fluid. Time-harmonic waves in the fluid are scattered by the elastic obstacle; the solution to this transmission problem is unique apart from the eigenpairs of the Jones eigenproblem.

Note that \autoref{eq:introjones2} together with the traction free condition in \autoref{eq:introjones3} constitute the usually accepted formulation of the eigenvalue problem for the Lam\'e operator with Neumann boundary conditions. It is well known that this problem has a countable set of eigenpairs (see, e.g. \cite{ref:babuskaosborn1991} for a 2D example). We remark that the existence of eigenpairs is independent of the domain shape in the sense that rigid motions are eigenfunctions associated with the eigenvalue zero as long as the problem in \autoref{eq:introjones} is well-defined. This is not the case for the Jones eigenproblem: the extra constraint on the normal trace of the displacement imposes geometrical conditions which may play an important role in the existence of eigenpairs on some domains. Indeed, the author in \cite{ref:harge1990} was able to exhibit that the eigenpairs of \autoref{eq:introjones} do not exist for most $C^\infty$ domains in 3D. However, it is not difficult to check that a 2D rotation satisfies the Jones eigenproblem with $\omega^2 = 0$ as eigenvalue (see \autoref{fig:solutionsball}) whenever $\Omega$ is a circle or its complement. This is also true for the sphere in 3D where rotations around the three directions $x_1$, $x_2$ and $x_3$ are eigenvectors associated with the eigenvalue $\omega^2 = 0$. These simple examples exhibit a strong connection between the shape and properties of the domain $\Omega$ and the existence of a spectrum for this problem.

It has been recently shown in \cite{ref:dominguez2019} that eigenpairs of \autoref{eq:introjones} do exist on general Lipschitz domains in 2D and 3D. It was also proven that the spectrum of this problem depends on the geometry of the domain: for an is an axisymmetric domain the eigenvalues are non-negative with rotations as  eigenvectors associated with $w = 0$; for an unbounded domain with at least two parallel faces as part of its boundary, its eigenvalues are non-negative and translations conform the eigenspace of $w = 0$; for general non-axisymmetric and bounded Lipschitz domains, the eigenvalues are strictly positive. In this paper, we are able to find eigenpairs for a weaker problem: one has existence of Jones eigenpairs if one puts the condition $\u\cdot \n = 0$ only on a non-empty part of the boundary with $(n-1)$-dimensional measure $\Sigma\subseteq\partial\Omega$. Although the geometrical properties of $\Sigma$ change in this case, we see that the zero eigenvalue is added to the spectrum when $\Sigma$ is either a flat face or a circle-shaped surface (around an axis of symmetry).
On the other hand, we introduce an eigenvalue problem where the condition on the zero normal trace on $\Sigma$ is changed by a zero tangential trace on $\Sigma$. We prove that, depending on the shape of $\Sigma$, we have a countable set of eigenpairs where the zero eigenvalue is added to the spectrum with rigid motions as associated eigenfunctions. As suggested by the Korn's inequality for vector fields with vanishing tangential trace, the eigenfunctions corresponding to the zero eigenvalue intimately depend on the shaped of $\Sigma$, as for the case of the Jones eigenfunctions. Nevertheless, the geometry conditions that the tangential trace imposes on $\Sigma$ are obviously different from what the normal trace imposes.

The rest of this paper is organized as follows: in \autoref{section:korns} we introduce some notation and provide a 
brief discussion on rigid motions (\autoref{subsection:notation} and \autoref{subsection:rigidmotions} respectively), 
to then state and prove the Korn's inequality for smooth enough vector fields on Lipschitz domains whose normal 
or tangential trace vanishes on part of the boundary (see \autoref{subsection:kornsnormal} and \autoref{subsection:kornstangent}). In \autoref{section:jones}, we first introduce the Jones eigenvalue 
problem by describing the fluid-structure interaction problem where this eigenproblem naturally appears (see \autoref{subsection:fluidstructure}). 
In \autoref{subsection:existencejones}, we use the proven Korn's inequality from 
\autoref{subsection:kornsnormal} to show the existence of Jones eigenpairs for Lipschitz domains in 2D and 3D. We further show in \autoref{subsection:variantjones} that eigenpairs of \autoref{eq:introfirstkorn} do exist when the normal trace condition on $\Sigma$ is replaced by the tangential trace. Finally, we comment in \autoref{subsection:variabledensity} about the extension of the studied eigenproblems to linearly elastic bodies with variable density.


%% file: korn.tex
\subsection{Some notation}\label{subsection:notation}
We begin this section by introducing some notation to be used throughout this paper. Given a Hilbert 
space $H$ of scalar 
fields, we denote by $\hh$ to the vector valued functions such that each scalar component belongs to $H$. 
Further, $\hhh$ is 
utilized to denote tensor fields whose each entry belong to $H$. Vector fields will be 
denoted with bold 
symbols whereas tensor fields are denoted with bold Greek letters. For an open domain $\Omega$ of 
$\rrr^n$, $n\in\nnn$, the 
space $W^{s,p}(\Omega)$ denotes the usual Sobolev space of scalar fields, for $s\in\rrr$ and $1<p<\infty$, with norm 
$\|\cdot\|_{s,p,\Omega}$. For vector fields, we use the notation $\ww^{s,p}(\Omega)$ with 
the corresponding norm simply denoted by $\|\cdot\|_{s,p,\Omega}$. In particular, the Hilbert space $H^s(\Omega)$ reduces to 
the usual Sobolev space $W^{s,2}(\Omega)$ with norm $\|\cdot\|_{s,\Omega} := \|\cdot\|_{s,2,\Omega}$. 
Whenever is well defined, the inner product in $H^s(\Omega)$ is $(\cdot,\cdot)_{s,\Omega}$, whereas 
$[\cdot,\cdot]_{s,\Omega}$ is the duality pairing between $\big(H^s(\Omega)\big)^*$ and $H^s(\Omega)$. The 
vector version of $H^s(\Omega)$ is denoted by $\hh^s(\Omega)$. In particular, we use the convention 
$H^0(\Omega) = L^2(\Omega)$ and $\hh^0(\Omega) = \ll^2(\Omega)$. On the boundary $\partial\Omega$ (or part of it), the Sobolev space $W^{s,p}(\partial\Omega)$ is define accordingly for values $s\in\rrr$ and $1<p<\infty$ (see, e.g., \cite{ref:mclean2000}), with $[\cdot,\cdot]_{s,p,\partial\Omega}$ denoting the duality pairing between $W^{s,p}(\partial\Omega)$ and its dual space. Between vectors, the operation $\a\cdot\b$ is the 
standard dot product with induced norm $\|\cdot\|$. In turn, for tensors $\bsigma,\,\btau$, the double dot product is 
the usual inner product for matrices which induces the Frobenius norm, that is $\bsigma:\btau := 
\tr(\btau^\transpose\bsigma)$. For measurable tensors, $\lll^p(\Omega)$ denotes the 
space of measurable tensors with finite and measurable tensor $p$-norm (Frobenius norm if $p=2$). 

For differential operators, $\nabla$ denotes the usual gradient operator acting on either a scalar field or a vector 
field. The divergence operator ``\div'' of a vector field reduces to the trace of its gradient, while the operator 
``\bfdiv'' acting on tensors stands for the usual divergence operator applied to each row of tensors. The rotation 
operator ``\curl'' denotes the rotation of a vector in 3D. However, a 2D version of this operator can be defined where 
\curl\, acts only in the $\hat z$ direction. In fact, note that the 3D rotation
\begin{align*}
 \curl\,\u := \left(\frac{\partial u_3}{\partial x_2} - \frac{\partial u_2}{\partial 
x_3}\right)\hat{x}_1 + 
\left(\frac{\partial u_1}{\partial x_3} - \frac{\partial u_3}{\partial x_1}\right)\hat{x}_2 + 
\left(\frac{\partial 
u_2}{\partial x_1} - \frac{\partial u_1}{\partial x_2}\right)\hat{x}_3,
\end{align*}
becomes $\curl\,\u = \left(\frac{\partial u_2}{\partial x_1} - \frac{\partial u_1}{\partial 
x_2}\right)\hat{x}_3$ 
in the 2D case, where $\u$ is extended as a vector with 3 entries, that is $\u := (u_1,u_2,0)^\transpose$. 
\begin{figure}[!ht]
\centering
\includegraphics[width = 1.0\textwidth, 
height=0.35\textheight]{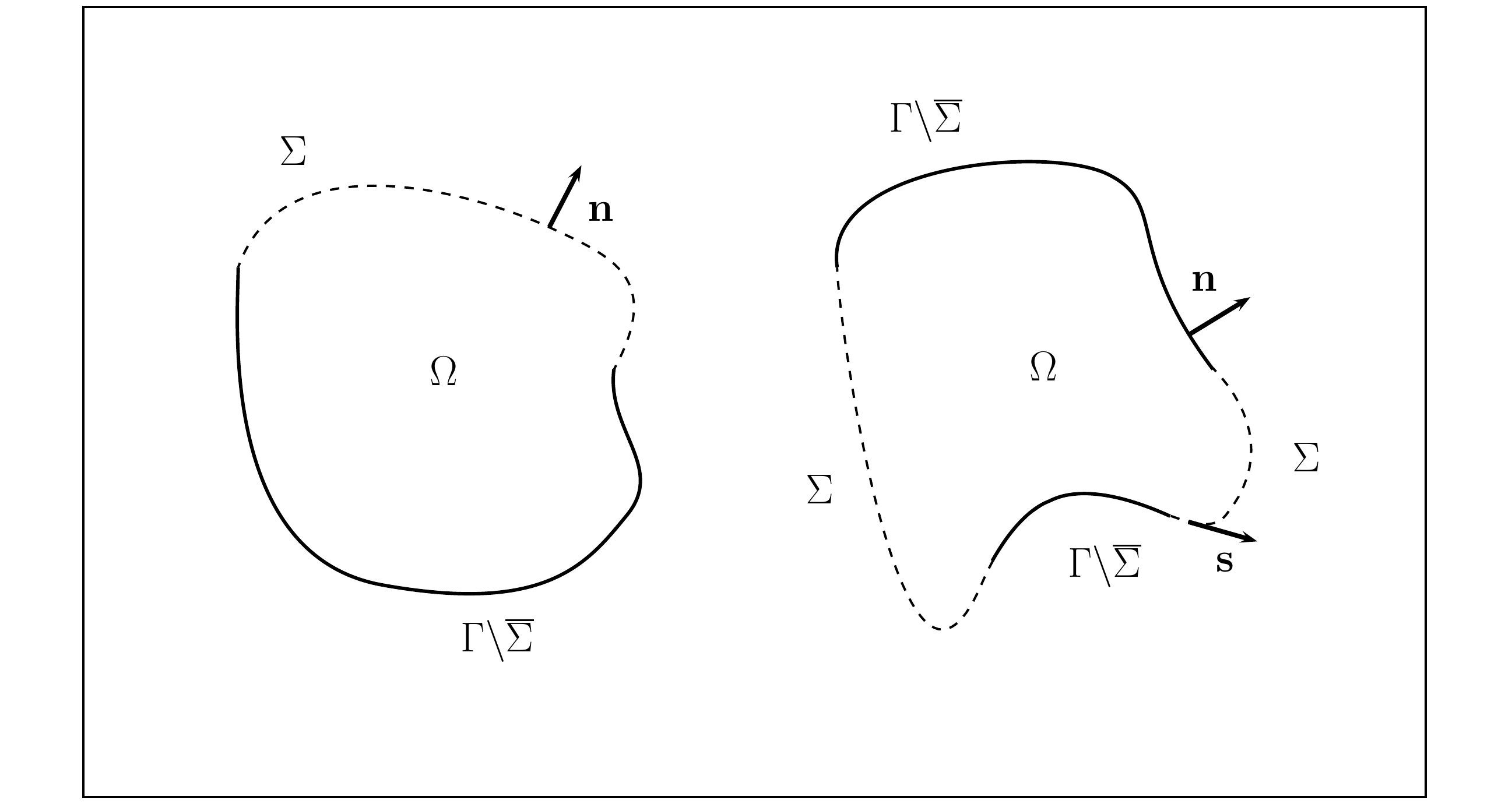}
\caption{Schematic of the domain in $\rrr^2$.}
\label{fig:setup}
\end{figure}
For an open and simply connected domain $\Omega$ of $\rrr^n$, we denote by $\n$ to the outer normal unit vector on the 
boundary $\Gamma:=\partial\Omega$. The tangent vector $\s$ can be defined as the cross product $\hat{x}_3\times\n$ (see \autoref{fig:setup}), where the normal $\n$ is extended to a 3D vector as $\n := (n_1,n_2,0)^\transpose$. Let us denote by $H(\div;\Omega)$ to the space of all vector fields in 
$\ll^2(\Omega)$ with divergence in $L^2(\Omega)$ The normal trace operator $\gamma_\n:H(\div;\Omega)\to 
H^{-1/2}(\Gamma)$ is bounded and linear with $\|\gamma_\n(\v)\|_{-1/2,\Gamma}\leq \|\v\|_{\div;\Omega}$ 
for all $\v\in H(\div;\Omega)$ (see, e.g. \cite{ref:gaticabook2014} for a detailed discussion on the normal trace in 
the space $H(\div;\Omega)$). The space $H^{-1/2}(\Gamma)$ is the the dual space of $H^{1/2}(\Gamma)$. For vectors in $\hh^1(\Omega)$, the operator $\gamma_\n$ can be identified with the trace 
operator
$\gamma_0:\hh^1(\Omega)\to \hh^{1/2}(\Gamma)$ (cf. \cite[Eq. (1.45)]{ref:gaticabook2014}) as follows
\begin{align*}
 [\gamma_\n(\v),q ]_{1/2,\Gamma} := \int_{\Gamma} \gamma_0(\v)\cdot\n\, q,\quad 
\forall\, q\in 
H^{1/2}(\Gamma).
\end{align*}
If $\Omega$ is a Lipschitz domain, then the unit normal vector $\n$ on $\Gamma$ belongs to $\ll^\infty(\Gamma)$ and thus $\gamma_0(\v)\cdot\n\in \ll^2(\Gamma)$, for all $\v\in\hh^1(\Omega)$. 

In turn, the tangential trace, $\gamma_\t$ is defined in terms of the trace operator as follows
\begin{align*}
    \gamma_\t(\v) := \gamma_0(\v) - (\gamma_0(\v)\cdot\n)\,\n\quad\forall\,\v\in\hh^1(\Omega).
\end{align*}
With these definitions, for each $\v\in\hh^1(\Omega)$ we have that
\begin{align*}
    \|\gamma_\t(\v)\|^2 &= \gamma_\t(\v)\cdot\gamma_\t(\v)\\
    &= \|\gamma_0(\v)\|^2 - 2\gamma_0(\v)\cdot(\gamma_0(\v)\cdot\n)\n + |\gamma_0(\v)\cdot\n|^2\\
    &= \|\gamma_0(\v)\|^2 - |\gamma_0(\v)\cdot\n|^2,
\end{align*}
that is
\begin{align*}
    \|\gamma_0(\v)\|^2 = |\gamma_0(\v)\cdot\n|^2 + \|\gamma_\t(\v)\|^2\quad\forall\,\v\in\hh^1(\Omega).
\end{align*}
Since $\hh^{1/2}(\Gamma)\subseteq \ll^2(\Gamma)$, the relation above implies that $\gamma_\t(\v)\in\ll^2(\Gamma)$, for all $\v\in\hh^1(\Omega)$. We also have that
\begin{align*}
    \|\gamma_0(\v)\cdot\n\|_{0,\Gamma}\leq\, \|\v\|_{1,\Omega},\quad\|\gamma_\t(\v)\|_{0,\Gamma}\leq\, \|\v\|_{1,\Omega} \quad \forall\,\v\in\hh^1(\Omega).
\end{align*}
If $\Sigma\subseteq \Gamma$ is a non-empty subset of the boundary of $\Omega$ with positive $(n-1)$-dimensional measure, the Sobolev space $H^{1/2}(\Sigma)$ contains all restrictions to $\Sigma$ of functions in 
$H^{1/2}(\Gamma)$ (see, e.g. \cite{ref:mclean2000} for a more detailed description of these spaces). The restriction of the trace operator, $\gamma_0(\cdot)|_\Sigma$ is well defined and allows us to define normal trace of elements in $\hh^1(\Omega)$. Following the reasoning to define the normal and tangential traces on $\Gamma$, their restrictions to $\Sigma$ are well defined as elements in $\ll^2(\Sigma)$, with
\begin{align}
    \|\gamma_0(\v)|_\Sigma\cdot\n|_\Sigma\|_{0,\Sigma}\leq\, \|\v\|_{1,\Omega},\quad\|\gamma_\t(\v)|_\Sigma\|_{0,\Sigma}\leq\, \|\v\|_{1,\Omega} \quad \forall\,\v\in\hh^1(\Omega).\label{eq:normaltangenth1}
\end{align}


Finally, we employ $\zero$ to denote the zero vector, tensor, or operator, depending on the context.

\subsection{Rigid motions}\label{subsection:rigidmotions}
As mentioned in \autoref{section:intro}, 
one needs to be aware of rigid motions of the domain. 
Let $\Omega$ be an open, bounded and simply connected domain in $\rrr^n$, $n\geq 2$. We define the space of all rigid 
motions of $\Omega$ as
\begin{align*}
 \rr\mm(\Omega):=\Big\{\v\in\ll^2(\Omega):\,\v(\x) = \b+\bb\x,\,\b\in\rrr^n,\,\bb^\transpose=-\bb,\,\x\in\Omega\Big\}.
\end{align*}
In this space, we identify two types of motions: pure rotations and translations. If $\rr(\Omega)$ and 
$\tt(\Omega)$ denote the space of pure rotations and pure translations of $\Omega$ respectively, then we have the 
following decomposition:
\begin{align*}
 \rr\mm(\Omega) = \rr(\Omega)+\tt(\Omega).
\end{align*}
It is well known that rigid motions are strain-energy free. In fact, let us define the strain tensor of a vector 
$\u\in\hh^1(\Omega)$ by 
\begin{align*}
 \bepsilon(\u) := \frac{1}{2}\left(\nabla\u+\nabla\u^\transpose\right).
\end{align*}
The strain tensor $\bepsilon(\cdot)$ is a linear and bounded operator from $\hh^1(\Omega)$ to $\lll^2(\Omega)$. 
In this sense, it is easy to show that null space of this operator exactly coincides with the space of rigid motions, 
that is
\begin{align}
 N(\bepsilon(\cdot)) = \rr\mm(\Omega).\label{eq:strain-rigid}
\end{align}
This implies that the Korn's inequality in \autoref{eq:introfirstkorn} cannot hold for arbitrary vector field in $\hh^1(\Omega)$, 
and therefore the strain tensor itself cannot define an equivalent norm in $\hh^1(\Omega)$.


\subsection{Korn's inequality and vanishing normal trace}\label{subsection:kornsnormal}
Let us consider an open, bounded and simply connected domain $\Omega$ in $\rrr^n$, $n\geq 2$, with Lipschitz 
boundary $\Gamma := \partial\Omega$. Let $\n$ denote the unit normal vector pointing out from $\Gamma$. As first shown 
by J.A. Nitsche \cite{ref:nitsche1981}, the {\it Korn's inequality} for vector fields in $\hh^1(\Omega)$ can be written 
as
\begin{align}
\|\nabla\v\|_{0,\Omega} \leq C\left(\|\bepsilon(\v)\|_{0,\Omega} + 
\|\v\|_{0,\Omega}\right),\quad\forall\,\v\in\hh^1(\Omega),\label{eq:kornsh1}
\end{align}
where $C>0$ is a constant depending only on $\Omega$. 
Let us consider a non-empty part of the boundary $\Sigma \subseteq \Gamma$ (possibly $\Sigma = \Gamma$) such that its 
$(n-1)$-dimensional measure is positive, i.e., $|\Sigma| > 0$. Define the Sobolev space
\begin{align}
 \hh^1_\n(\Omega;\Sigma) := \Big\{\v\in\hh^1(\Omega):\,\gamma_0(\v)\cdot\n = 0\quad\text{a.e. on $\Sigma$}\Big\}.\label{eq:h1normal}
\end{align}
This space is equipped with the usual $\hh^1$-norm. The continuity of $\gamma_0(\cdot)\cdot\n|_\Sigma$ implies the closeness 
of $\hh^1_\n(\Omega;\Sigma)$ in $\hh^1(\Omega)$. Assuming a zero normal trace on part of the boundary may not 
exclude rigid motions from $\hh^1_\n(\Omega;\Sigma)$. 
To obtain the Korn's inequality in \autoref{eq:introfirstkorn}, some properties of the domain $\Omega$ must be fixed to ensure the kernel of 
$\bepsilon(\cdot)$ in $\hh^1_\n(\Omega;\Sigma)$ is the trivial space. Indeed, in the 2D case, if $\Sigma$ is contained in a 
straight line of $\Gamma$, then the spaces $\tt(\Omega)$ and $\hh^1_\n(\Omega;\Sigma)$ have a non-trivial intersection. On the 
other hand, if $\Sigma$ is contained in the surface of a ball, then $\rr(\Omega)\cap\hh^1_\n(\Omega;\Sigma)$ has 
at least dimension 1. This means that $\Omega$ can support rigid motions $\v\in\rr\mm(\Omega)$ that are tangential to $\Sigma$ as long as $\gamma_0(\v)\in\gen\{\n|_\Sigma\}^\perp$. We summarize these properties in the following result.
 

\begin{theorem}\label{result:rmnormal}
Let $\v:=\bb\x+\b$, $\x\in\rrr^n$, be a non-zero rigid motion such that $b_{ij} = -b_{ji} = b\neq 0$ if $i = 1$ and $j = 2$, and $b_{ij} = 0$ otherwise, and $b_i = 0$ for all $i=3,\ldots,n$.
%
Let $f:\rrr^n\to\rrr$ be a Lipschitz continuous function defining $\Sigma$ almost everywhere. Then, the condition $\gamma_0(\v)\cdot\n = 0$ holds on $\Sigma$ if and only if the function $f$ can be written as
\begin{align}
    f(\x) = \frac{b}{2}(x_1^2+x_2^2) + b_1x_2 - b_2x_1 + g(x_3,\ldots,x_n),\label{eq:deflipschitzmap}
\end{align}
for some Lipschitz continuous function $g:\rrr^{n-2}\to\rrr$. If $n=2$, this function is simply a constant.
\end{theorem}
\begin{proof}
Note that the given rigid motion can be written as
\begin{align*}
\v(\x) := (bx_2+b_1,-bx_1+b_2,0,\ldots,0)^\transpose.
\end{align*}
The unit normal vector on $\Sigma$ is then
\begin{align*}
    \n(\x) = \frac{\nabla f(\x)}{\|\nabla f(\x)\|}\quad\text{a.e. $\x\in\Sigma$}.
\end{align*}
The condition $\gamma_0(\v)\cdot\n = 0$ on $\Sigma$ implies that $\v$ and $\n$ are mutually orthogonal in the $\rrr^n$-inner product, that is $\v$ lies in the plane generated by $\n$. Equivalently, this means that $\n$ belongs to the plane generated by $\v$. 
From the vanishing normal trace condition we have the following equations, which yield almost everywhere in $\Sigma$,
\begin{align}
    (bx_2+b_1)\frac{\partial f(\x)}{\partial x_1} + (-bx_1+b_2)\frac{\partial f(\x)}{\partial x_2} = 0.\label{eq:normalcondition}
\end{align}





From the condition above and the constraint $\gamma_0(\v)\cdot\n = 0$ on $\Sigma$, the normal vector becomes
\begin{align*}
    \n(\x) = \mp \frac{(bx_1-b_2,bx_2+b_1,\g(\x))^\transpose}{\sqrt{(-bx_1+b_2)^2 + (bx_2+b_1)^2+\|\g(\x)\|^2}},
\end{align*}
where the vector-valued function $\g:\rrr^n\to\rrr^{n-2}$ is
\begin{align*}
    \g(\x):= \frac{bx_1-b_2}{\frac{\partial f}{\partial x_1}}\left(\frac{\partial f}{\partial x_3},\ldots,\frac{\partial f}{\partial x_n}\right)^\transpose.
\end{align*}
From here we see that $\frac{\partial f}{\partial x_1} = bx_1-b_2$, $\frac{\partial f}{\partial x_1} = bx_2+b_1$, and thus $\g(\x) := \left(\frac{\partial f}{\partial x_3},\ldots,\frac{\partial f}{\partial x_n}\right)^\transpose$,
and therefore completing the proof.

\end{proof}
We remark that this result outlines the important and dependence on the domain to be able to support a rigid motion which is tangential to the boundary $\Sigma$. We see that the shape of this part of the boundary is forced by the form of the rigid motion. If a rigid motion is to have more non-zero entries (more than two), then more conditions are added to the function $f$ as described above and therefore one expects to have a different pattern in the part of the boundary where one wants the selected rigid motion to be tangential. Moreover, we can identify that the general idea behind the previous result is to show there is a variety of Lipschitz domains which support tangential rigid motions, even when they are tangent only on $\Sigma$. This dependence is translated to the existence of a plane generated by the unit normal vector $\n$ on $\Sigma$ such that at least some rigid motions in $\rr\mm(\Omega)$ belong to this plane. For example, for $n=2$, the proof of the result just above shows that the rigid motion $\v:=(-x_2,x_1)^\transpose$ belongs to the plane generated by the normal vector $\n(\x) := \frac{(x_1,x_2)^\transpose}{\sqrt{c}}$, where $c$ is the constant given in the proof (assumed to be positive now). This says that the boundary $\Sigma$ belongs to the arc of a circle of radius $\sqrt{c}$ and centred at the origin. This can be extended to domains in 3D, where the rigid motion $\v(\x) := (x_2,-x_1,0)^\transpose$ is supported by $\Omega$ if and only in the normal on $\Sigma$ is $\n(\x):=(x_1,x_2,g'(x_3))^\transpose$, for some Lipschitz continuous function $g$. For a fixed $x_3\in\rrr$ such that $\x\in\Sigma$, the equation $x_1^2+x_2^2 = g(x_3)^2$ represents the arc of a circle of radius $|g(x_3)|$ and centre at the origin. If we are to add more rigid motions to this domain, the the function $g$ can be specified. In fact, if one wants the rotation $\v:= (x_3,0,-x_1)^\transpose$ to satisfy the condition $\gamma_0(\v)\cdot\n = 0$ on $\Sigma$, then we must have that $g(x_3) = \frac{x_3^2}{2}+c$, for some constant $c\in\rrr$. The normal vector becomes $\n(\x) = (x_1,x_2,x_3)^\transpose$, which implies that $\Sigma$ defines a patch of a sphere. Furthermore, we see that the rotation $\v(\x) := (0,-x_3,x_2)^\transpose$ automatically satisfies the vanishing normal trace condition on $\Sigma$.

We also remark that given a rigid motion $\v$ as defined in the statement of \autoref{result:rmnormal}, then one can construct a Lipschitz continuous function $f:\rrr^n\to\rrr$, depending on $f$, such that $f(\x) = 0$ defines a Lipschitz continuous surface in $\rrr^n$, $\Sigma$. In fact, the unit normal vector can be defined as
\begin{align*}
    \n(\x) := \pm\frac{\rr(\x)\,\v(\x)}{\|\rr(\x)\,\v(\x)\|},
\end{align*}
where $\rr(\x)$ is a rotation matrix such that $\v$ and $\n$ are mutually orthogonal at $\x\in\Sigma$. This says that, given a rigid motion $\v$, we can always find a domain $\Omega$ with a Lipschitz continuous patch $\Sigma$ of the boundary such that $\v\cdot\n = 0$ a.e. on $\Sigma$. This comes from the fact that the Lipschitz function defining the boundary $\Sigma$ for this domain cannot be written as in the form given by \autoref{result:rmnormal}.

 

The converse is, in general, not true. Not every domain can support a rigid motion. For example, if $\Omega$ is the unit square in $\rrr^2$ and $\Sigma$ consists on the union of the lines $x_1 = 0$ and $x_2 = 0$, then it is not hard to show that no rigid motions $\v$ satisfying the condition $\gamma_0(\v)\cdot\n = 0$ along $\Sigma$. More generally, for a given domain $\Omega$ and Lipschitz continuous subset $\Sigma$ of the boundary $\Gamma$, one has the following.
\begin{align}
    \dim(\rr\mm(\Omega)\cap\hh^1_\n(\Omega;\Sigma)) = \dim(\gen\{\n|_\Sigma\}^\perp),\label{eq:spannormal}
\end{align}
where the orthogonal complement is taken with respect to the usual inner product in $\rrr^n$. However, as the example presented above, the space $\gen\{\n|_\Sigma\}^\perp$ may be the trivial space for some shapes of $\Sigma$. In this sense, applying this to the example above we see that the unit normal vector on $\Sigma$ is $(-1,0)^\transpose$ and $(0,-1)^\transpose$, which together form a basis for $\rrr^2$. This indicates that $\gen\{\n|_\Sigma\}^\perp$ is the trivial space.





The Korn's inequality for vector fields in $\hh^1_\n(\Omega;\Sigma)$ is proven in the next theorem.

\begin{theorem}\label{result:kornsh1normal}
 Assume $\Omega$ is an open, bounded and simply connected domain in $\mathbb{R}^n$ with Lipschitz boundary 
$\Gamma:=\partial\Omega$. Let $\Sigma\subseteq\Gamma$ with positive $(n-1)$-dimensional measure such that $\gen\{\n|_\Sigma\}^\perp$ is the trivial space.
Then, there 
exists a constant $C > 0$ such that
\begin{align}
\|\u\|_{1,\Omega} \leq 
C\|\bepsilon(\u)\|_{0,\Omega},\quad\forall\,\u\in\hh^1_\n(\Omega;\Sigma).\label{eq:1stkorn}
\end{align}
\end{theorem}
\begin{proof}
 By contradiction, suppose we can find $\u_k\in\hh^1_\n(\Omega;\Sigma)$ such that
\begin{align*}
  \|\u_k\|_{1,\Omega} = 1,\quad \|\bepsilon(\u_k)\|_{0,\Omega} < \frac{1}{k},\quad\forall\,k\in\nnn.
 \end{align*}
 Since $\{\u_k\}$ is bounded in $\hh^1(\Omega)$, we know that there is a vector field $\u\in\hh^1(\Omega)$ and a 
subsequence $\{\u_{k_l}\}$ of $\{\u_k\}$ such that $\u_{k_l}\to\u$ weakly in $\hh^1(\Omega)$. Also, using that 
the inclusion $\hh^1(\Omega)\hookrightarrow\ll^2(\Omega)$ is compact, we have that $\u_{k_l}\to\u$ strongly in 
$\ll^2(\Omega)$. Moreover, note that $\bepsilon(\u_{k_l})\to {\bf 0}$ in $\ll^2(\Omega)$. Using the Korn's 
inequality in \autoref{eq:kornsh1} we obtain
\begin{align*}
 \|\u_{k_j}-\u_{k_l}\|_{1,\Omega} \leq C\left(\|\bepsilon(\u_{k_j}-\u_{k_l})\|_{0,\Omega} + 
\|\u_{k_j}-\u_{k_l}\|_{0,\Omega}\right),
\end{align*}
that is, $\{\u_{k_l}\}$ is a Cauchy sequence in $\hh^1(\Omega)$. The completeness of $\hh^1(\Omega)$ 
and the weak 
convergence of $\{\u_{k_l}\}$ in $\hh^1(\Omega)$ imply that $\u_{k_l}\to\u$ strongly in 
$\hh^1(\Omega)$, and $\|\u\|_{1,\Omega} 
= 1$. Furthermore, the closeness of $\hh^1_\n(\Omega;\Sigma)$ in $\hh^1(\Omega)$ implies that $\u\in\hh^1_\n(\Omega;\Sigma)$. 
In turn, we see that
\begin{align*}
 \|\bepsilon(\u_{k_l})-\bepsilon(\u)\|_{0,\Omega} \leq \frac{1}{2}\|\u_{k_l}-\u\|_{1,\Omega}\to 0,
\end{align*}
which says that $\bepsilon(\u) = {\bf 0}$. Therefore $\u$ is a rigid motion in $\rrr^n$ with $\gamma_\n(\u) = 0$ a.e. 
on $\Sigma$. Since $\Sigma$ is 
such that $\gen\{\n|_\Sigma\}^\perp$ is the trivial space, the charactierization in \autoref{eq:spannormal}
implies the intersection $\rr\mm(\Omega)\cap\hh^1_\n(\Omega;\Sigma)$ is the trivial space, concluding that $\u = {\bf 0}$, 
which is a contradiction since $\|\u\|_{1,\Omega} = 1$.
\end{proof}

The same proof remains true in the case $\Sigma$ exactly coincides with $\Gamma$. Nevertheless, the shape of the entire boundary is defined in this case by the normal trace as this condition is carried out along the 
whole boundary. For pure translations we can see that the space $\tt(\Omega)\cap\hh^1_\n(\Omega;\Gamma)$ would be 
non-trivial if and only if $\Gamma$ consists of at least one plane in $\rrr^n$. Note that the boundness of the domain is lost for this case. On the other hand, 
$\rr(\Omega)\cap\hh^1_\n(\Omega;\Gamma)$ would have non-zero vectors in more cases. To identify these domains, we consider the following definition concerning symmetries of the shape of the domain.
\begin{definition}[adopted from \cite{ref:bauer2016-1,ref:desvillettes2002}]\label{def:axisymmetricdomain}
 An open domain $U\subseteq\rrr^n$ is axisymmetric if there is a non-zero vector $\r\in\rr(U)$ such that $\r\cdot\n = 0$ a.e. on $\partial U$.
\end{definition}
With this definition we see that the only axisymmetric domains in $\rrr^2$ are the circle and its complement. 
Nonetheless, in higher dimensions the number of axisymmetric domains becomes very large. For example, any solid of revolution 
of a Lipschitz continuous function defined on a bounded interval in $\rrr$ would be axisymmetric in $\rrr^3$. In this 
manner, whenever the domain $\Omega$ is axisymmetric, the space $\rr(\Omega)\cap\hh^1_\n(\Omega;\Gamma)$ would have 
dimension at least one. Indeed, as shown in \cite{ref:bauer2016-2}, the inequality in \autoref{eq:1stkorn} holds for 
non-axisymmetric Lipschitz domains provided $\Sigma$ coincides exactly with the boundary $\Gamma$.

\subsection{Case of vanishing tangential trace}\label{subsection:kornstangent}
One can derive a similar conclusion as the one given in \autoref{result:kornsh1normal} but for vectors with a zero tangential trace on part of the boundary. As in the previous section, let $\Omega$ be a bounded and simply connected Lipschitz domain in $\rrr^n$ and let $\Sigma$ be a non-empty part of the boundary $\Gamma:=\partial\Omega$ with positive $(n-1)$-dimensional measure. Define the space
\begin{align*}
 \hh^1_\s(\Omega;\Sigma) := \Big\{\v\in\hh^1(\Omega):\,\gamma_\t(\v) = \zero\quad\text{a.e. on $\Sigma$}\Big\}.
\end{align*}
In this space we consider the usual $\hh^1$-norm. With the definition and properties of the tangential trace in 
$\hh^1(\Omega)$ we can show that $\hh^1_\s(\Omega;\Sigma)$ is a closed subspace of $\hh^1(\Omega)$. 

The case $\Sigma = \Gamma$ was proven in \cite{ref:bauer2016-2}, where the authors showed that the Korn's 
inequality in \autoref{eq:1stkorn} holds 
for any vector field in $\hh^1_\s(\Omega;\Sigma)$ with no extra assumptions on the geometry of $\Omega$. However, in case 
$\Sigma$ is strictly included in $\Gamma$, the shape of $\Sigma$ plays an important role in the validity of 
\autoref{eq:1stkorn}.

\begin{theorem}\label{result:rmtangent}
Under the same assumptions as in \autoref{result:rmnormal}, the rigid motion $\v$ satisfies the condition $\gamma_\t(\v) = \zero$ on $\Sigma$ if and only if the function $f$ can be written as
\begin{align*}
    f(\x) := b_1x_2-b_2x_1 - g(x_3,\ldots,x_3),
\end{align*}
for some Lipschitz continuous function $g:\rrr^{n-2}\to\rrr$. If $n=2$, then $g$ is simply a constant function.
\end{theorem}
\begin{proof}
The proof follows the same essential steps to those presented in the proof of \autoref{result:rmnormal}. However, here the function $f$ satisfies the condition
\begin{align*}
    (-bx_1+b_2)\frac{\partial f(\x)}{\partial x_1} - (bx_2+b_1)\frac{\partial f(\x)}{\partial x_2} = 0,
\end{align*}
for almost every $\x\in\Sigma$. This gives the following form of the unit normal vector on $\Sigma$
\begin{align*}
    \n(\x) := \frac{\v(\x)}{\|\v(\x)\|}\quad\x\in\Sigma.
\end{align*}
This completes the proof as $\gamma_\t(\v) = \zero$ on $\Sigma$ with this choice of the normal vector.

\end{proof}
As for the case of tangential rigid motions on $\Sigma$ shown in \autoref{result:rmnormal}, the result above provides the simplest case in which rigid motions shape the form of the boundary normal rigid motion on $\Sigma$ are considered. If one needs to add more rigid motions then the shape of $\Sigma$ must change accordingly to be able to satisfy the tangential condition for all the rigid motions. In essence, we have the following characterization of the intersection between the space of rigid motions and $\hh^1_\s(\Omega;\Sigma)$
\begin{align*}
    \dim(\rr\mm(\Omega)\cap\hh^1_\s(\Omega;\Sigma)) = \dim(\Pi_{\n|_\Sigma}^\perp),
\end{align*}
where $\Pi_{\n|_\Sigma}$ is the plane generated by $\n|_\Sigma$, which is tangent to $\Sigma$. The orthogonal complement is taken in the $\rrr^n$ usual inner product, which must hold almost everywhere on $\Sigma$.

The corresponding Korn's inequality for vector fields in $\hh^1_\s(\Omega;\Sigma)$ is given next.
\begin{theorem}\label{result:kornsh1tangent}
 Assume $\Omega$ is an open, bounded and simply connected domain of $\rrr^n$, $n\geq 2$ with Lipschitz boundary 
$\Gamma:=\partial\Omega$. Let $\Sigma$ be a subset of $\Gamma$ with positive $(n-1)$-dimensional measure such that $\Pi_{\n|_\Sigma}^\perp$ is the trivial space. Then, 
there is a positive constant $c>0$, such that
\begin{align*}
 \|\u\|_{1,\Omega} \leq c\|\bepsilon(\u)\|_{0,\Omega},\quad\forall\,\u\in\hh^1_\s(\Omega;\Sigma).
\end{align*}
\end{theorem}
\begin{proof}
 The proof follows from the same steps given in the proof of \autoref{result:kornsh1normal} and the use of 
\autoref{result:rmtangent} to derive the necessary contradiction.
\end{proof}


In the forthcoming section we introduce the {\it Jones eigenproblem}, where elastic waves with traction 
free 
condition are constrained to have a vanishing normal trace on the boundary. This extra condition means that this 
eigenvalue problem is over-determined; we may not have eigenpairs for this problem in some situations 
(see, e.g. \cite{ref:harge1990}). However, with the help of \autoref{result:kornsh1normal}, we are able to show that, 
in most of the cases for Lipschitz domains, there is a 
complete set of eigenfunctions with non-negative eigenvalues.

%% file: jones.tex

\subsection{Fluid-structure interaction problem}\label{subsection:fluidstructure}
As discussed in \autoref{section:intro}, the Jones eigenproblem was originally described within the context of a 
fluid-structure interaction problem. 
Consider a bounded, simply connected domain 
$\Omega_s\subseteq\mathbb{R}^n$ with boundary $\Gamma_s 
:= \partial\Omega_s$ representing an isotropic and linearly elastic body in $\mathbb{R}^n$. This body is assumed to be
immersed in a compressible inviscid fluid occupying the region $\Omega_f := \mathbb{R}^n\backslash\bar{\Omega}_s$. See 
\autoref{fig:schematic} for a schematic of this situation.

\begin{figure}[!ht]
\centering
\includegraphics[width = .75\textwidth, 
height=0.25\textheight]{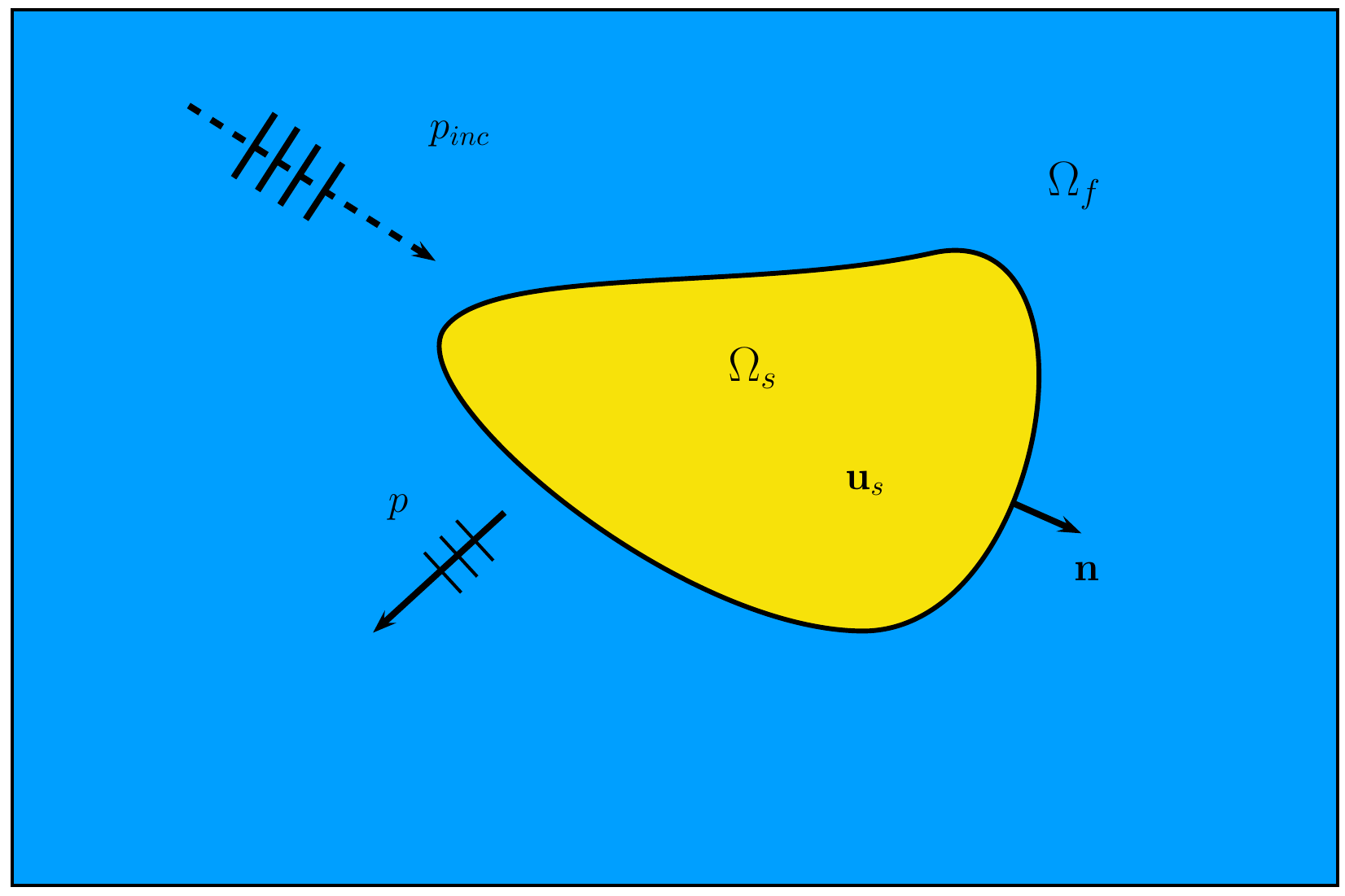}
\caption{Schematic of the fluid-structure interaction problem.}
\label{fig:schematic}
\end{figure}
Note that 
the bounded part of the boundary of $\Omega_f$, 
$\Gamma_f:=\partial\Omega_f$ coincides with the boundary of the (bounded region) $\Omega_s$. For simplicity we write 
$\Gamma := \Gamma_f = \Gamma_s$.

The parameters describing the elastic properties of $\Omega_s$ are the so-called Lam\'e constants $\mu>0$ and 
$\lambda\in\mathbb{R}$, satisfying the condition
\begin{align}
 \lambda+\left(\frac{2}{n}\right)\mu>0\label{eq:lamecondition}
\end{align}
One fluid-structure interaction problem of interest concerns the situation when the fields are time-harmonic, allowing 
us to factor out the time-dependence and consider the problem in the frequency domain. Using standard  interface 
conditions coupling the pressure in the fluid $p$ and the elastic displacement in the solid $\u$, the fluid-solid 
interaction problem in the frequency domain reads: given a prescribed pressure $q\in L^2(\Omega_f)$,a volume force $\g\in\ll^2(\Omega_s)$, and an 
incident pressure $p_{\rm inc}\in H^1(\Omega_f)$, we want to locate a pressure field $p$ in $\Omega_f$ and elastic deformations $\u$ of 
$\Omega_s$, 
satisfying
\begin{subequations}\label{eq:coupled-problem}
 \begin{align}
 &\Delta p + \left(\frac{\omega^2}{c^2}\right)p = g,\,\,\text{on $\Omega_f$},\quad-\rho \omega^2\u 
- \bfdiv\,\bm\sigma(\u) = \g,\,\,\text{in $\Omega_s$},\\
 &-(p + p_{\rm inc})\n = \bm\sigma(\u)\n,\quad \frac{\partial}{\partial\n}( p + p_{\rm inc}) = 
\rho\omega^2\u\cdot\n,\,\,\text{on $\Gamma$},\\
 &\frac{\partial p}{\partial r} - i\left(\frac{\omega}{c}\right) p = o(1/r),\,\,\text{as 
$r:=\|\x\|\to\infty$}.
\end{align}
\end{subequations}
The parameter {$c$} is the constant speed of the sound in the fluid, $\rho$ is the density of the solid (assumed to be constant), and $\bm\sigma$ is the usual Cauchy tensor for linear elasticity. This is defined in terms of the strain tensor $\bepsilon(\u)$ as
\begin{align*}
	\bm\sigma(\u) := 2\mu\bepsilon(\u) + \lambda\,\tr(\bepsilon(\u))\ii,\quad\text{in $\Omega_s$.}
\end{align*}
This is a commonly accepted formulation for time-harmonic fluid-solid interaction problems involving inviscid flow, 
see, for example, \cite{ref:hsiao2000, ref:hsiao2017, ref:huttunen2008}. The system in 
\autoref{eq:coupled-problem} is known to possess a non-trivial kernel under certain situations. As discussed in 
\cite{ref:jones1983}, this problem lacks a unique solution whenever $\u$ is a non-trivial solution of the homogeneous 
problem:
\begin{align}
 -\bfdiv\,\bm\sigma(\u) = \rho \omega^2\u,\,\,\text{in $\Omega_s$},\quad\bm\sigma(\u)\n = {\bf 
0},\quad\u\cdot\n = 0,\,\,\text{on $\Gamma$}.\label{eq:jones-modes}
\end{align}
The pair $(\omega^2,\u)$ solving this eigenvalue problem is a Jones eigenpair \cite{ref:jones1983}. The homogeneous problem 
for the 
displacements can be viewed as the usual 
eigenvalue problem for linear elasticity with traction free boundary condition, plus the extra constraint on the normal 
trace of $\u$ along the boundary. Therefore, we may consider this as an over-determined problem. We know that there is a 
countable 
number of eigenmodes for linear 
elasticity 
with free traction given reasonable assumptions on $\Gamma$  (see \cite{ref:babuskaosborn1991} and references therein). 
The extra 
condition  $\u\cdot \n=0$ on the boundary  plays an important role in the existence of the zero eigenvalue of 
\autoref{eq:jones-modes}. All of these 
properties are discussed in detailed in the next section.

\subsection{Jones eigenpairs}\label{subsection:joneseps}
Let $\Omega$ be an open, bounded and simply connected domain in $\rrr^n$, $n\in\{2,3\}$, with Lipschitz boundary 
$\Gamma:=\partial\Omega$. We denote by $\n$ and $\s$ the normal and tangential unit vectors on $\Gamma$. The 
normal 
vector is chosen to point out from $\Omega$. 
Assume $\u:\Omega\to\rrr^n$ denotes the displacement vector of small 
deformations of an isotropic elastic material occupying the domain $\Omega$ of constant density $\rho>0$. The Jones eigenvalue problem reads: find a non-zero displacement $\u$ and a frequency $\omega\in\ccc$ such 
that 
\begin{subequations}\label{eq:jonesmain}
\begin{align}
\mu \Delta \u + (\lambda+\mu)\nabla(\div\,\u) + \rho\omega^2\u = \zero\quad \text{in 
$\Omega$,}\label{eq:jonesmain1}\\
\t(\u) = \zero,\quad \u\cdot\n = 0, \quad \text{on 
$\Gamma$,}\label{eq:jonesmain2}
\end{align}
\end{subequations}
where $\t(\u)$ is the traction operator on $\Gamma$, defined as
\begin{align*}
\t(\u) :=  2\mu \frac{\partial \u}{\partial \n} + \lambda(\div\,\u)\,\n + \mu(\n\times\curl\,\u),
\end{align*}
and the constants $\lambda$ and $\mu$ are the Lam\'e parameters as described in the previous section, and satisfy the condition given in \autoref{eq:lamecondition}. 

This formulation of the Jones eigenproblem is equivalent to that given by \autoref{eq:jones-modes}. Indeed, using the vector Laplacian operator, we see that
\begin{align*}
	{\rm\bf div}\,\bm\sigma(\u) = \mu\Delta\u + (\lambda + \mu)\nabla({\rm div}\,\u)\quad\text{on 
$\Omega$}.
\end{align*}
The traction operator $\t(\u)$ then becomes $\t(\u) = \bsigma(\u)\n$ on $\Gamma$.

In the present manuscript we analyse the existence of eigenpairs of slightly different problem, which can be reduced to the original formulation of the Jones eigenvalue problem. Let $\Sigma\subseteq \Gamma$ be a non-empty set such that $|\Sigma|>0$. We are interested in displacements $\u$ of $\Omega$ and frequencies $\omega\in\rrr$ for which \autoref{eq:jonesmain1} is satisfied along with its free traction boundary condition $\t(\u) = \zero$ along the boundary $\Gamma$ (see first condition in \autoref{eq:jonesmain2}). As an extra constraint, we only impose the condition $\u\cdot\n = 0$ on the part of the boundary $\Sigma$. Concretely, we want to find eigenpairs $(\omega^2,\u)\in\rrr\times\hh^1(\Omega)$ solving the following eigenproblem:
\begin{subequations}
\begin{align*}
\mu \Delta \u + (\lambda+\mu)\nabla(\div\,\u) + \rho\omega^2\u = \zero\quad \text{in 
$\Omega$,}\\
\t(\u) = \zero\quad \text{on $\Gamma$,}\quad \u\cdot\n = 0, \quad \text{on $\Sigma$,}
\end{align*}
\end{subequations}
It is clear that this problem coincides with the formulation of the Jones eigenproblem if $\Sigma = \Gamma$. Within this manuscript, eigenpairs solving this problem are simply called {\it Jones eigenpairs}. We can again re-formulate the problem above with the use of the Cauchy stress tensor as follows
\begin{subequations}\label{eq:main}
\begin{align}
\bfdiv\,\bsigma(\u) + \rho\omega^2\u = \zero\quad \text{in 
$\Omega$,}\label{eq:main1}\\
\bsigma(\u) = \zero\quad \text{on $\Gamma$,}\quad \u\cdot\n = 0, \quad \text{on $\Sigma$,}\label{eq:main2}
\end{align}
\end{subequations}

In the next section we prove that Jones eigenpairs do exists whenever the domain is Lipschitz. We further show that this is true even when the zero normal trace is assumed only in the non-empty part $\Sigma$ of the boundary of the domain. Nonetheless, there are many cases for which rigid motions are part of the spectrum of the problem. As described in \autoref{section:korns}, we see that the number of eigenfunctions associated with the zero Jones eigenvalue increases as we increase the dimension of the problem and changes as we modify the shape of $\Sigma$.

\subsection{Existence of generalised Jones eigenpairs}\label{subsection:existencejones}
Throughout this section we assume that $\Omega$ is a bounded and simply connected Lipschitz domain of $\rrr^n$, with 
$n\in\{2,3\}$. Let $\Sigma\subseteq\Gamma:=\partial\Omega$ be a non-empty subset such that $|\Sigma|>0$. In general, analytic 
solutions of eigenvalue problems may not be simple to find (if possible) explicitly on domains 
other than the rectangle or the ball \cite{ref:natroshvili2005,ref:}. Alternatively, 
numerical methods can be utilized to approximate the spectrum of linear operators defined on more general domains. A 
particular 
choice 
is to derive a weak formulation to characterize and show the existence of eigenpairs. With this approach, we seek 
eigenpairs satisfying a generalized eigenvalue problem through the use of sesquilinear forms. We can apply this approach to the Jones 
eigenvalue problem in \autoref{eq:main}. Consider its equivalent 
form \autoref{eq:main} in terms of the strain tensor $\bepsilon(\cdot)$ to obtain the following 
weak formulation: find eigenpairs $\omega^2\in\rrr$, $\u\in\hh^1_\n(\Omega;\Sigma)$ such that
\begin{align}
 a(\u,\v) = \omega^2 b(\u,\v),\quad\forall\,\v\in\hh^1_\n(\Omega;\Sigma),\label{eq:weakform}
\end{align}
where the space $\hh^1_\n(\Omega;\Sigma)$ is defined in \autoref{eq:h1normal} (cf. \autoref{subsection:kornsnormal}), and the 
sesquilinear forms $a(\cdot,\cdot)$ and $b(\cdot,\cdot)$ are given by
\begin{align*}
 a(\u,\v) := (\bsigma(\u),\bepsilon(\v))_{0,\Omega},\quad b(\u,\v) := \rho\, (\u,\v)_{0,\Omega},\quad \forall \,
\u,\,\v\in\hh^1_\n(\Omega;\Sigma).
\end{align*}
This formulation has been obtained by multiplying equation \autoref{eq:main1} with 
$\v\in\hh^1_\n(\Omega;\Sigma)$ and then integrating by parts. At that point, the traction free condition in \autoref{eq:main2} was used to derive the formulation in \autoref{eq:weakform}. Observe that the 
bilinear form 
$b(\cdot,\cdot)$ is well defined in $\ll^2(\Omega)$, and it induces an equivalent norm in this space.

We now define the solution 
operator $T:\ll^2(\Omega)\to\hh^1_\n(\Omega;\Sigma)$ of \autoref{eq:weakform} as $T(\f) = \u$, where $\u$ and $\f$ solve 
the source problem
\begin{align}
 a(\u,\v) = b(\f,\v),\quad\forall\,\v\in\hh^1_\n(\Omega;\Sigma).\label{eq:source}
\end{align}
The goal is to relate the spectrum of the operator $T$ with the eigenpairs of \autoref{eq:weakform}. 
In this way, we can see that $T(\u) = \kappa\u$, $\kappa\neq0$, is a solution of \autoref{eq:source} if and only if 
$\omega^2 = \frac{1}{\kappa}$ and $\u$ solves \autoref{eq:weakform}. For now, it is clear that $T$ is a linear 
operator. 
Nonetheless, 
further properties of this linear operator are needed to establish a more precise description of 
the spectrum of $T$, and they can be shown only if the sesquilinear forms posses additional properties.

Note that $a(\cdot,\cdot)$ and $b(\cdot,\cdot)$ are bilinear forms (real sesquilinear), they are both positive, with 
$b(\u,\u) > 0$ for any non-zero $\u\in\hh^1_\n(\Omega;\Sigma)$. The Rayleigh quotient shows that
\begin{align*}
 \omega^2 = \frac{a(\u,\u)}{b(\u,\u)},\quad \u\in\hh^1_\n(\Omega;\Sigma),\,\, \u\neq\zero.
\end{align*}
This implies that all eigenvalues of \autoref{eq:weakform} (equivalently \autoref{eq:main}) are non-negative. 
The following result allow us to show the continuity of the operator $T$.
\begin{theorem}\label{result:aelliptic}
 Assume $\Omega$ does not satisfy any of the properties in \autoref{result:rmnormal}. Then, there is a constant 
$c>0$, such that
\begin{align*}
 a(\u,\u) \geq c\, \|\u\|_{1,\Omega}^2,\quad \forall\,\u\in\hh^1_\n(\Omega;\Sigma).
\end{align*}
\end{theorem}
\begin{proof}
 From the definition of the bilinear form $a(\cdot,\cdot)$, we can derive the bound
\begin{align*}
  a(\u,\u) \geq \min\Big\{2\mu,n\left(\lambda + 
\frac{2}{n}\mu\right)\Big\}\|\bepsilon(\u)\|_{0,\Omega}^2,\quad\forall\,\u\in\hh^1_\n(\Omega;\Sigma).
 \end{align*}
Now, since $\Omega$ does not satisfy any of the properties listed in \autoref{result:rmnormal} (cf. 
\autoref{subsection:kornsnormal}), the Korn's 
inequality provided by \autoref{result:kornsh1normal} gives the existence of a constant $C>0$ such that 
$\|\bepsilon(\u)\|_{0,\Omega}\geq C\,\|\u\|_{1,\Omega}$, for any vector field $\u$ in $\hh^1_\n(\Omega;\Sigma)$. Thus, we get
\begin{align*}
 a(\u,\u) \geq c\, \|\u\|_{1,\Omega}^2,
\end{align*}
with constant $c := C^2 \min\Big\{2\mu,n\left(\lambda + \frac{2}{n}\mu\right)\Big\}$.
\end{proof}
Having this result, we can show that $T$ is a bounded linear operator, with 
\begin{align*}
    \|T(\u)\|_{1,\Omega}\leq\,\frac{\rho}{c}\,\|\u\|_{0,\Omega}\quad\forall\,\u\in\hh^1_\n(\Omega;\Sigma),
\end{align*}
where $c>0$ is 
defined as in the proof of the previous result. In addition, the compactness 
of the inclusion $\hh^1_\n(\Omega;\Sigma)\hookrightarrow\ll^2(\Omega)$ shows that the restriction of $T$ to $\hh^1_\n(\Omega;\Sigma)$, say $\bar T$, is a compact 
operator from $\hh^1_\n(\Omega;\Sigma)$ onto $\hh^1_\n(\Omega;\Sigma)$. Finally, the symmetry of the bilinear 
forms $a(\cdot,\cdot)$ and $b(\cdot,\cdot)$ implies that $\bar T$ is a self-adjoint operator with respect to the inner product induced by 
the bilinear form $a(\cdot,\cdot)$. Therefore, using the well-known Spectral Theorem for bounded, linear, compact and 
self-adjoint operators, we have the following result.
\begin{theorem}\label{result:spectrum}
 Assume $\Omega$ does not satisfy any of the properties given in \autoref{result:rmnormal}
. Then the operator $\bar T$ has a 
countable spectrum $\{\kappa_l\}_{l\in\nnn}\subseteq (0,\|\bar T\|)$ such that 
$\kappa_l \to 0$ as $l$ goes to infinity, with eigenfunctions $\{\u_l\}$ in $\hh^1_\n(\Omega;\Sigma)$, 
mutually orthogonal in the $\ll^2$-inner product.
\end{theorem}

Using the spectrum of $\bar T$ and the relation $\kappa = \frac{1}{\omega^2}$, we have that $\omega^2_l 
:= \frac{1}{\kappa_l}$ form a countable sequence of strictly positive eigenvalues of 
\autoref{eq:weakform} such that $w_l^2 \to +\infty$ as $l\to+\infty$, with eigenfunctions $\u_l\in\hh^1_\n(\Omega;\Sigma)$, for all $l\in\nnn$. Even though no closed form of
these eigenpairs is known, numerical methods can provide approximations to them.

In case $\Omega$ satisfies one of the properties in \autoref{result:rmnormal}, as discuss in \autoref{section:korns}, 
the first Korn's 
inequality given in \autoref{result:kornsh1normal} implies that $\omega^2 = 0$ is an eigenvalue of 
\autoref{eq:weakform} with associated eigenvalues lying in $\rr\mm(\Omega)$. This implies that the coercivity of the 
bilinear form $a$ cannot hold in $\hh^1_\n(\Omega;\Sigma)$, and thus the necessary properties of the operator $T$ are not 
longer guaranteed. To overcome this issue, we can shift the formulation in \autoref{eq:weakform} to get the new 
formulation: find $\tilde\omega^2\in\rrr$ and $\u\in\hh^1_\n(\Omega;\Sigma)$, $\u\neq \zero$, such that
\begin{align}
 \tilde a(\u,\v) = \tilde\omega^2b(\u,\v),\quad\forall\,\v\in\hh^1_\n(\Omega;\Sigma),\label{eq:shiftedweakform}
\end{align}
where $\tilde a(\u,\v) := a(\u,\v) + b(\u,\v)$, and $\tilde\omega^2 := \omega^2+1$. Using the 
equivalent formulation of the generalised Jones eigenvalue problem in \autoref{eq:main}, one can easily get that
\begin{align*}
 \tilde a(\u,\v) \geq \min\{\mu,\rho\}\|\u\|_{1,\Omega}^2,\quad\forall\,\u\in\hh^1_\n(\Omega;\Sigma).
\end{align*}
Consequently, one can define a solution operator $\tilde T:\hh^1_\n(\Omega;\Sigma)\to\hh^1_\n(\Omega;\Sigma)$ as 
in \autoref{eq:source} by replacing $a(\cdot,\cdot)$ with $\tilde a(\cdot,\cdot)$. Note that the eigenfunctions associated with the eigenvalue $\tilde\omega^2 = 1$ lie in the space of rigid motions, $\rr\mm(\Omega)$. Since this space is finite dimensional, the restriction of $\tilde T$ to $\hh^1_\n(\Omega;\Sigma)$, $\hat T := T|_{\hh^1_\n(\Omega;\Sigma)}:\hh^1_\n(\Omega;\Sigma)\to \hh^1_\n(\Omega;\Sigma)$, is continuous, compact and self-adjoint, with $\|\hat T\|=\frac{\rho}{\min\{\mu,\rho\}}$. Then
\autoref{result:spectrum} also applies to this operator: the spectrum of $\hat T$ consists of eigenvalues 
$\{\tilde\kappa_l\}_{l\in\nnn}\subseteq (0,1)\cup\{1\}$ and eigenfunctions $\{\tilde\u_l\}\subseteq \hh^1_\n(\Omega;\Sigma)$ which are
orthogonal in the 
$\ll^2$-inner product. We have that $\tilde\omega_l^2 = \frac{1}{\tilde\kappa_l}$ is the countable sequence of strictly positive eigenvalues of \autoref{eq:shiftedweakform}, with lower 
bound $\tilde\omega^2 = 1$ and such that  $\tilde\omega_l^2\to+\infty$ as $l$ goes to infinity.

When $\Sigma = \Gamma$, \autoref{result:aelliptic} and \autoref{result:spectrum} remain valid. Here 
$\Omega$ needs to be a non axisymmetric Lipschitz domain. The last result summarizes the properties of $T$ in for this 
case.
\begin{theorem}\label{result:jones-spectrum}
 Assume $\Sigma = \Gamma$. If
 \begin{enumerate}
    \item $\Omega$ is a non-axisymmetric domain, then the operator $T|_{\hh^1_\n(\Omega;\Sigma)}$ has a 
countable spectrum $\{\kappa_l\}_{n\in\nnn}\subseteq (0,1)$ such that 
$\kappa_l \to 0$ as $l$ goes to infinity, with eigenfunctions $\{\u_l\}$ in $\hh^1_\n(\Omega;\Sigma)$, 
mutually orthogonal in the $\ll^2$-inner product.
  \item $\Omega$ is an axisymmetric domain, then $\omega_0 = 0$ is also an eigenvalue of \autoref{eq:jonesmain} with 
associated eigenspace $\rr(\Omega)$, apart from the countable sequence $\{\omega^2_l\}_{l\in\nnn}$ of strictly positive eigenvalues.
  \item $\Omega$ is an unbounded domain with its boundary consisting of at least one plane in $\rrr^n$, then $\omega_0 
= 0$ is an eigenvalue of \autoref{eq:jonesmain} with associated eigenfunctions belonging to $\tt(\Omega)$.
 \end{enumerate}
\end{theorem}
\begin{proof}
 Parts 1 can be derived by combining \cite[Lemma 9 and Theorem 18]{ref:bauer2016-2} or by using the Korn's 
inequality given in \autoref{result:kornsh1normal} for $\Sigma = \Gamma$
. For part 2, it is 
straightforward to see that there is a rotation that is tangential to $\Gamma$; one can take a rotation around the axis of symmetry of the domain. Then the pair 
$\omega_0=0$ and $\u_0\in\rr(\Omega)$ would satisfy the Jones eigenvalue problem in \autoref{eq:jonesmain}. 

Finally, for part 3, if the boundary of $\Omega$ consists at least one plane, then the normal vector on 
$\partial\Omega$ is a unit vector in $\rrr^n$. Then, the basis which defines the plane obtained form this normal is contained in 
$\tt(\Omega)$. Thus the pair $\omega_0 = 0$ and $\u_0\in\tt(\Omega)$, where $\u_0$ is orthogonal to the 
normal vector on the boundary, satisfies the Jones eigenvalue problem in this case as well.
\end{proof}

We comment that the last part in the previous result the existence of a countable spectrum cannot be guaranteed. This comes from the fact that the compactness of the corresponding solution operator is crucial to obtain this property as part of the Spectral Theorem. It is known that for unbounded domains the compactness is not true in general (see \cite{ref:demengel2012} for a good example on this matter).

\subsection{A variant of the Jones eigenproblem}\label{subsection:variantjones}
We have seen in the previous sections that the validity of the Korn's inequality (cf. \autoref{result:kornsh1normal}) provides the existence of eigenpairs of the Jones eigenvalue problem in \autoref{eq:jonesmain}. \autoref{result:kornsh1tangent} suggests that a similar eigenproblem would then have a countable set of eigenpairs. Let $\Omega$ be a Lipschitz domain in $\rrr^n$, $n\geq 2$, with boundary $\Gamma:=\partial\Omega$, and let $\Sigma$ be a non-empty subset of $\Gamma$ with $|\Sigma|>0$. Assume we now want to locate eigenpairs $(\omega,\u)$ of the Lam\'e operator which are purely orthogonal to $\Sigma$, that is, we need to find small displacements $\u$ and frequencies $\omega\in\ccc$ such that
\begin{subequations}\label{eq:jonestangent}
\begin{align}
    -\bfdiv\,\bsigma(\u) = \rho\, \omega^2 \u\quad\text{in $\Omega$},\\
    \bsigma(\u)\,\n = \zero\quad\text{on $\Gamma$},\,\gamma_\t(\u) = \zero\quad\text{on $\Sigma$}.
\end{align}
\end{subequations}
It was given in \cite{ref:dominguez2019} the analytical expressions of the true eigenpairs of the Jones eigenproblem in \autoref{eq:jones-modes} on the rectangle $[0,a]\times[0,b]$. Based of these, one can easily obtain analytical solutions to the eigenproblem above. In fact, if $\Sigma = \Gamma$, then the condition $\gamma_\t(\cdot) = \zero$ on $\Gamma$ gives the following eigenvalues and eigenfunctions
\begin{subequations}
\begin{align*}
&\u_s := \left((a\ell)\cos\Big(\frac{m\pi x}{a}\Big)\sin\Big(\frac{\ell\pi y}{b}\Big), -(bm)
\sin\Big(\frac{m\pi x}{a}\Big)\cos\Big(\frac{\ell\pi y}{b}\Big)\right)^\transpose,\\ 
&w_s^2 := \frac{\mu\pi^2}{\rho}\left(\frac{m^2}{a^2}+\frac{\ell^2}{b^2}\right),\quad m,\ell = 
1,2,\ldots,
\end{align*}
\end{subequations}
and 
\begin{subequations}
\begin{align*}
&\u_p := \left((bm) \cos\Big(\frac{m\pi x}{a}\Big)\sin\Big(\frac{\ell\pi y}{b}\Big), (a\ell)\sin\Big(\frac{m\pi 
x}{a}\Big)\cos\Big(\frac{\ell\pi y}{b}\Big)\right)^\transpose,\\
&w_p^2 := \frac{(\lambda + 2\mu)\pi^2}{\rho}\left(\frac{m^2}{a^2}+\frac{\ell^2}{b^2}\right),\quad m,\ell = 
0,1,\ldots,\quad m+\ell >0.
\end{align*}
\end{subequations}
This suggests that, as for the Jones eigenproblem, there might be a large class of domains that can support eigenpairs of \autoref{eq:jonestangent}.
For this eigenvalue problem we consider the following formulation: find $\u\in\hh^1_\t(\Omega;\Sigma)$ and $\omega\in\ccc$ such that
\begin{align}
    a(\u,\v) = \omega^2 b(\u,\v),\quad\forall\,\v\in\hh^1_\t(\Omega;\Sigma),\label{eq:jonestangent3}
\end{align}
where the bilinear forms $a(\cdot,\cdot)$ and $b(\cdot,\cdot)$ are defined as in the previous section. As for the Jones eigenproblem, the eigenvalues of \autoref{eq:jonestangent} are real and nonnegative. Let us define the solution operator $T:\ll^2(\Omega)\to\hh^1_\t(\Omega;\Sigma)$ as $T(\f) = \u$, where for a given data $\f\in\ll^2(\Omega)$, we are to find $\u\in \hh^1_\t(\Omega;\Sigma)$ such that
\begin{align*}
    a(\u,\v) = b(\f,\v),\quad\forall\,\v\in\hh^1_\t(\Omega;\Sigma),
\end{align*}
The Korn's inequality stated in \autoref{result:kornsh1tangent} implies, together with the Lax-Milgram lemma, that in case $\Omega$ does not satisfy the condition listed in \autoref{result:rmtangent}, there is a unique solution $\u\in\hh^1_\t(\Omega)$ of the problem above. Also, there is a constant $C>0$ such that
\begin{align*}
    \|\u\|_{1,\Omega}\leq\, C\|\f\|_{0,\Omega}.
\end{align*}
This means that the operator $T$ is bounded in both the $\hh^1$- and the $\ll^2$-norms with $\|T\| = C$. In addition, the compact embedding $\hh^1_\t(\Omega;\Sigma) \hookrightarrow \ll^2(\Omega)$ implies that $T$ is a compact operator. Finally, the symmetry of the bilinear forms $a(\cdot,\cdot)$ and $b(\cdot,\cdot)$ gives the symmetry of $T$. Altogether, we come to the conclusion that $T$ possesses a countable set of eigenpairs $\kappa_l\in(0,1)$ and $\u_l\in\hh^1_\t(\Omega;\Sigma)$. Note that the eigenvalues of \autoref{eq:jonestangent3} are given by $\omega_l^2 = \frac{1}{\kappa_l}$, and the corresponding eigenfunctions are the same as those of $T$.

However, if $\Omega$ satisfy at least one of the conditions in \autoref{result:rmtangent}, then we know that rigid motions are a solution of \autoref{eq:jonestangent3} with $\omega_0 = 0$. Obviously, not every rigid motion is an eigenfunction for a given domain $\Omega$. The following result summarizes the properties of the operator $T$.
\begin{theorem}\label{result:jonestangent-spectrum}
 The spectrum of $T|_{\hh^1_\t(\Omega;\Sigma)}$ is given by eigenvalues $\{\kappa_l\}_{l\in\nnn}$ with eigenfunctions $\{\u_l\}_{l\in\nnn}\in\hh^1_\t(\Omega;\Sigma)$. If 
 \begin{enumerate}
     \item the domain $\Omega$ is such that $\Sigma$ does not does not satisfy the conditions in \autoref{result:rmtangent}, then $\omega_l^2>0$;
     \item $\Sigma$ satisfies one of the conditions listed in \autoref{result:rmtangent}, then $\omega_0 = 0$ is added to the countable spectrum described above, with corresponding eigenfunctions lying in $\rr\mm(\Omega)\cap\hh^1_\t(\Omega;\Sigma)$.
 \end{enumerate} 
\end{theorem}

\subsection{Linearly elastic bodies with variable density}\label{subsection:variabledensity}
In many realistic applications the density of the elastic body may be variable. For this situation we see that the key properties used in the proof of the existence of spectrum of the Jones eigenproblem in \autoref{eq:jonesmain} and its variant defined in \autoref{eq:jonestangent} remain true. However, the orthogonality properties of the eigenfunctions changes: eigenfunctions corresponding to different eigenvalues are orthogonal in the weighted $\ll^2$-inner product, with the variable density as the weight. We end this manuscript with the theorem stating this case.
\begin{theorem}\label{result:variabledensity}
Assume the density of the elastic body $\rho$ belongs to $L^\infty(\Omega)$. Then \autoref{result:spectrum} and \autoref{result:jonestangent-spectrum} remain true. However, eigenfunctions corresponding to different eigenvalues are orthogonal in the weighted inner product $(\rho\,\cdot,\cdot)_{0,\Omega}$.
\end{theorem}